\documentclass[12pt]{amsart}
\usepackage{epsfig,afterpage}
\usepackage{amsthm}
\usepackage{latexsym,amssymb,amsfonts}
\usepackage{amsmath}
\usepackage{graphicx}
\usepackage{comment}
\numberwithin{equation}{section}
\def\R{\mathbb R}
\def\C{\mathbb C}
\def\N{\mathbb N}
\def\Z{\mathbb Z}
\def\BB{\mathcal B}
\def\CC{\mathcal C}
\def\FF{\mathcal F}
\def\JJ{\mathcal J}
\def\KK{\mathcal K}
\def\HH{\mathcal H}
\def\XX{\mathcal X}

\def\re{\operatorname{Re}}
\def\im{\operatorname{Im}}
\def\diam{\operatorname{diam}}
\def\supp{\operatorname{supp}}
\def\area{\operatorname{area}}

\newtheorem{theorem}{Theorem}[section]
\newtheorem{lemma}{Lemma}[section]
\theoremstyle{remark}
\newtheorem*{remark}{Remark}

\begin{document}
\title[Hausdorff measure of hairs]{Hausdorff measure of hairs without endpoints in the exponential family}
\author{Walter Bergweiler}
\address{Mathematisches Seminar, Christian-Albrechts-Universit\"at zu Kiel,
24098 Kiel, Germany}
\email{bergweiler@math.uni-kiel.de}
\author{Jun Wang}
\address{School of Mathematical Sciences, Fudan University,
Shanghai 200433, P.~R.\ China}
\email{majwang@fudan.edu.cn}
\date{}
\begin{abstract}
Devaney and Krych showed that for $0<\lambda<1/e$ the Julia set of $\lambda e^z$
consists of pairwise disjoint curves, called hairs, which connect finite points,
called the endpoints of the hairs, with~$\infty$.
McMullen showed that the Julia set has Hausdorff dimension $2$ and Karpi\'nska showed
that the set of hairs without endpoints has Hausdorff dimension~$1$.
We study for which gauge functions the Hausdorff measure of the
set of hairs without endpoints is finite.
\end{abstract}
\maketitle

\section{Introduction and main results}
The \emph{Fatou set} $\FF(f)$ of a transcendental entire function $f$ is defined as the set of all
$z\in\C$ where the iterates $f^n$ of $f$ form a normal family. Its complement
$\JJ(f)=\C\backslash \FF(f)$ is called the \emph{Julia set}. These sets are the main objects studied in
complex dynamics; see~\cite{Bergweiler1993} and~\cite{Schleicher2010} for an introduction
to transcendental dynamics.

In some sense, the exponential functions $E_\lambda(z)=\lambda e^z$, with $\lambda \in\C\backslash\{0\}$,
are the ``simplest'' transcendental entire functions, and thus the dynamics of these functions have
been thoroughly studied; see~\cite{Devaney2010} for a survey, as well as, e.g., \cite{Rempe2006,Schleicher2003a}.

We mention some of the results that have been obtained.
Here we restrict to the case that $0<\lambda<1/e$, even though some of the results
discussed below hold more generally.
In the following we suppress the index $\lambda$ and write $E$ instead of~$E_\lambda$.
For $\lambda$ satisfying the above condition the function $E$ has two
real fixed points $\alpha$ and $\beta$ satisfying $\alpha< 1<\beta$, with $\alpha$
attracting and $\beta$ repelling.

Devaney and Krych~\cite[p.~50]{Devaney1984} proved $\FF(E)$ is equal to the attracting
basin of $\alpha$ and that $\JJ(E)$ consists of uncountably many pairwise disjoint curves
connecting a point in $\C$, called the endpoint of the curve, with $\infty$.
These curves are called \emph{hairs} (or dynamic rays).
McMullen~\cite[Theorem~1.2]{McMullen1987} proved that $\JJ(E)$ has Hausdorff dimension~$2$.
Let $\CC$ be the set of endpoints of the hairs that form $\JJ(E)$.
Karpi\'nska~\cite[Theorem~1.1]{Karpinska1999} proved the surprising
result that  $\JJ(E)\backslash \CC$ has Hausdorff dimension~$1$.
Of course, together with McMullen's result this implies that $\CC$ has  Hausdorff dimension~$2$,
a result she had proved already earlier~\cite[Theorem~1]{Karpinska1999a}.

McMullen remarked that $\JJ(E)$ not only has Hausdorff dimension~$2$, but that
in fact the Hausdorff measure $\HH^h(\JJ(E))$ of $\JJ(E)$ with respect
to the gauge function $h(t)=t^2/\log^m(1/t)$ is infinite, for any iterate
$\log^m$ of the logarithm; see section~\ref{hausdorff} below for the definition
of Hausdorff measure and Hausdorff dimension.
A very precise description of the gauge functions $h$ for which
$\HH^h(\JJ(E))=\infty$ was given by Peter~\cite{Peter2010}.

The purpose of this paper is to study for which gauge functions the
Hausdorff measure of $\JJ(E)\backslash\CC$ is finite or infinite.
Our first result is the following.
\begin{theorem}\label{thm2}
Let $s>1$.
Then $\mathcal{H}^h(\JJ(E)\backslash \mathcal{C})=0$
for $h(t)=t/(\log(1/t))^s$.
\end{theorem}
Our estimates in the opposite direction -- as well as the description
of $\JJ(E)\backslash \mathcal{C}$ in Theorems~\ref{thm5} and~\ref{thm4}
below that is used in the proofs -- involve fractional iterates.

In order to state these results, note that $E'(\beta)=E(\beta)=\beta$
so that the multiplier of the fixed point $\beta$ is also $\beta$.
It is a standard  result in complex dynamics that Schr\"oder's functional equation
\begin{equation}\label{schroeder}
S(\beta z)=E(S(z))
\end{equation}
has a solution $S$ which is holomorphic in a neighborhood $U$ of $0$ and satisfies $S(0)=\beta$ and $S'(0)=1$.
Since $\beta>1$, the equation~\eqref{schroeder} allows to extend $S$ to an entire function
by putting
\begin{equation}\label{schroeder2}
S(z)=E^k(S(z/\beta^k)),
\end{equation}
with $k$ so large that $z/\beta^k\in U$.

It is easy to see that $S$ is real on the real axis.
Since $S'(0)=1$ we have $S'(x)>0$ for all $x$ in $U\cap \R$, if $U$ is sufficiently small.
As $S'(x)=(E^k)'(S(x/\beta^k))S'(x/\beta^k)/\beta^k$ we see that in fact $S'(x)>0$ for all $x\in\R$.
Thus $S$ is increasing on $\R$.
It follows easily from~\eqref{schroeder2} that $S(x)\to\alpha$ as $x\to -\infty$ while
$S(x)\to\infty$ as $x\to\infty$. Thus $S\colon \R \to (\alpha,\infty)$ is bijective.

By $S^{-1}$ we denote the inverse of the restriction of $S$ to these intervals.
For $r\in \R$ the \emph{fractional iterates} $E^r\colon [\alpha,\infty)\to [\alpha,\infty)$
are then defined by $E^r(\alpha)=\alpha$ and
\begin{equation} \label{fraciter}
E^r(x)=S(\beta^{r}S^{-1}(x)).
\end{equation}
It follows easily from~\eqref{schroeder} that this coincides with the usual definition
of the iterates $E^r$ if $r\in\N$.
We also note that $E^r\circ E^s=E^s\circ E^r=E^{r+s}$ for $r,s\in\R$.
Moreover, $E^{-1}$ is the inverse function of $E\colon [\alpha,\infty)\to [\alpha,\infty)$.

We put $L=E^{-1}$ so that $L(x)=\log x-\log \lambda$.
The fractional iterates of $L$ are given by $L^r=E^{-r}$.
\begin{theorem}\label{thm1}
Let $s>1$.
Then $\mathcal{H}^h(\JJ(E)\backslash \mathcal{C})=\infty$ for $h(t)=t/L^s(1/t)$.
\end{theorem}

We make some remarks about the proofs.
An important ingredient in Karpi\'nska's proof that $\JJ(E)\backslash \CC$ has
Hausdorff dimension~$1$ was her result that if
\[
\Omega=\left\{z\in\C\colon \re z\geq M
\text{ and } |\im z|\leq (\re z)^\varepsilon\right\},
\]
with $M,\varepsilon>0$, and if $z\in \JJ(E)\backslash \CC$, then $E^k(z)\in\Omega$ for all large~$k$.
She then proved that, for sufficiently large~$M$,
the set of all $z$ with $E^k(z)\in\Omega$ for all $k\in\N$ has Hausdorff dimension
at most $1+\varepsilon$.

Karpi\'nska and Urba\'nski~\cite{Karpinska2006} also considered the dimension of certain
subsets of  $\JJ(E)$. While their result is not stated this way, it essentially says that with
\[
\Omega=\left\{z\in\C\colon \re z\geq M
\text{ and }
|\im z|\leq \frac{\re z}{(\log \re z)^\varepsilon}
\right\}
\]
the set of all $z$ with $E^k(z)\in\Omega$ for all $k\in\N$
and $\re E^k(z)\to\infty$ as $k\to\infty$
has Hausdorff dimension $(2+\varepsilon)/(1+\varepsilon)$.
We mention that the above domain $\Omega$ also appears in Stallard's construction~\cite{Stallard2000}
of entire functions whose Julia set has preassigned Hausdorff dimension in $(1,2)$.

We shall adapt the methods of Karpi\'nska and Urba\'nski to prove the following result.
\begin{theorem}\label{thm3}
Let $x_0>\beta$ and let $\psi\colon [x_0,\infty)\to (0,\infty)$ be an increasing function.
Suppose that $\psi(x)\to\infty$, $\psi(2x)=O(\psi(x))$ and $\psi(x)=o(x)$ as $x\to\infty$. 
\[
\Omega_\psi=\left\{z\in\C\colon \re z>x_0 \text{ and } |\im z|< \psi(\re z)\right\}
\]
and
\[
\XX=\XX(x_0,\psi)=\left\{z\in\C\colon
\re E^k(z)> E^k(x_0)
\text{ and }
E^k(z)\in \Omega_\psi \text{ for all }k\in\N \right\}.
\]
Let $h\colon(0,t_0)\to(0,\infty)$ be a gauge function of the form $h(t)=t/p(1/t)$ with
some increasing function $p\colon(1/t_0,\infty)\to(0,\infty)$ such that $t\to t\,p(1/t)$ is increasing
and let $\delta>0$.
Then:
\begin{itemize}
\item[$(i)$]
If
\begin{equation}\label{condp}
p\!\left(\frac{t(\log t)^{1+\delta}}{\psi(t)}\right)\leq \psi(\log t)
\end{equation}
for large $t$, then $\mathcal{H}^h(\XX)=\infty$.
\item[$(ii)$]
If
\begin{equation}\label{condp2}
p\!\left(\frac{t\log t}{\psi(t)}\right)\geq (\log t)^{1+\delta}
\end{equation}
for large $t$, then $\mathcal{H}^h(\XX)=0$.
\end{itemize}
\end{theorem}
Theorems~\ref{thm2} and~\ref{thm1} will follow
from Theorem~\ref{thm3} and the following two results.
\begin{theorem}\label{thm5}
Let $x_0>\beta$ and $z\in \JJ(E)\backslash \CC$. Then there exist
$\varepsilon>0$ and $k\in\N$ such that $E^k(z)\in \XX(x_0,L^\varepsilon)$.
\end{theorem}
\begin{theorem}\label{thm4}
Let $x_0>\beta$ and $\varepsilon>0$.
Then $\XX(x_0,L^\varepsilon)\subset \JJ(E)\backslash \CC$.
\end{theorem}
This paper is organized as follows. In section~\ref{hausdorff}
we recall the definition and properties of Hausdorff measure and Hausdorff dimension
and in section~\ref{fractional} we discuss fractional iterates of $E$ and $L$ in more detail than
in this introduction.
In section~\ref{hairs} we first recall some results about hairs and their endpoints
and then prove Theorems~\ref{thm5} and~\ref{thm4}.
Section~\ref{proofthm3} consists of the proof of Theorem~\ref{thm3}
while Theorems~\ref{thm2} and~\ref{thm1} are proved in section~\ref{proofthm21}.

\section{Hausdorff measure and Hausdorff dimension}
\label{hausdorff}
We recall the definition of Hausdorff measure and Hausdorff dimension; see
Falconer's book~\cite{Falconer1990} for more details.
For $A\subset \R^m$ we denote by $\diam A$ the (Euclidean) diameter of~$A$.
We denote the open ball of radius $r$ around a point $z\in \R^m$
by $D(z,r)$.
(We will only be concerned with the case $m=2$ so that $D(z,r)$ is a disk.)

Let $t_0>0$. An increasing, continuous function $h\colon (0,t_0)\to (0,\infty)$ which satisfies
$\lim_{t\to 0} h(t)=0$ is called a \emph{gauge function} (or \emph{dimension function}).
For  a subset $A$ of $\R^m$ and $\delta>0$ a sequence $(A_k)$ of
subsets of $\R^m$ is called a $\delta$-\emph{cover} of $A$ if $\diam A_k <\delta$ for all $k\in\N$ and
$$A\subset\bigcup_{k=1}^{\infty}A_k.$$
For a gauge
function $h$ we put
$$H_h^{\delta}(A)=\inf\left\{\sum_{k=1}^{\infty}h(\diam A_k)\colon (A_k)\text{ is }\delta\text{-cover of }A\right\}.$$
Note that $H_h^\delta(A)$ is a non-increasing function of $\delta$.  Thus the limit
$$\HH^h(A)=\lim_{\delta\to 0}H_h^{\delta}(A)$$
exists. It is called the \emph{Hausdorff measure} of $A$ with respect to the gauge function~$h$.

It may happen that $\sum_{k=1}^\infty h(\diam A_k)$ diverges for all $\delta$-covers $(A_k)$,
in which case we have $H_h^\delta(A)=\infty$ and thus $\HH^h(A)=\infty$.

In the special case that $h$ has the form $h(t)=t^s$ for some $s>0$, we call $\HH^h(A)$ the
$s$-\emph{dimensional Hausdorff measure}. There exists $d\geq 0$ such that
$\HH^{t^s}(A)=\infty$ for $0<s<d$ and $\HH^{t^s}(A)=0$ for $s>d$. This value $d$ is
called the \emph{Hausdorff dimension} of~$A$.

An important tool to estimate the Hausdorff measure and Hausdorff dimension from
below is the following result \cite[Theorem 7.6.1]{Przytycki2010}, which
is a part of Frostman's lemma and also known as the mass distribution principle.
\begin{lemma}\label{th21}
Let $h$ be a gauge function and $A\subset\R^m$. If there exists a Borel probability measure
$\mu$ supported on $A$ such that
$$\lim_{r\to 0}\frac{\mu(D(z,r))}{h(r)}=0\quad \text{for all }z\in A,$$
then $\HH^h(A)=\infty$.
\end{lemma}
In order to estimate the Hausdorff measure from above, we will use the following result.
\begin{lemma}\label{th22}
Let $A\subset \R^m$ and let $h$ be a gauge function.
Suppose that for all $x\in A$ and $\delta,\varepsilon>0$,
there exists $\rho(x)\in(0,1)$, $d(x)\in(0,\delta)$ and $N(x)\in\N$
satisfying $N(x)h(d(x))\leq \varepsilon\cdot\rho(x)^m$ such that $D(x,\rho(x))\cap A$ can be covered by
$N(x)$ sets of diameter at most $d(x)$. Then $\HH^h(A)=0$.
\end{lemma}
A very similar result for Hausdorff dimension can be found in~\cite[Lemma 5.2]{Bergweiler2010}.
Lemma~\ref{th22} can be proved by the same argument, but for completeness we include
the proof. As in~\cite{Bergweiler2010} we will use the following result~\cite[Lemma~4.8]{Falconer1990}.
\begin{lemma} \label{lemmafalc}
Let $K\subset \R^m$ be bounded, $R>0$ and
$\rho\colon K\to(0,R]$. Then there exists an at most countable subset $L$
of $K$ such that
\[
D(x,\rho(x))\cap D(y,\rho(y))=\emptyset
\quad\text{for } x,y\in L, \;   x\neq y,
\]
and
\[
\bigcup_{x\in K} D(x,\rho(x))\subset \bigcup_{x\in L} D(x,4\rho(x)) .
\]
\end{lemma}
\begin{proof}[Proof of Lemma~\ref{th22}]
Let $K$ be a bounded subset of $A$ and choose $R>0$ such that $K\subset D(0,R)$ and let
$\delta,\varepsilon>0$. Noting that
\[
K\subset \bigcup_{x\in K} D\!\left(x,\tfrac14\rho(x)\right)
\]
we deduce from Lemma~\ref{lemmafalc} that their exists an at most countable subset $L$
of $K$
such that
\[
K\subset \bigcup_{x\in L} D(x,\rho(x))
\]
while
\[
D\!\left(x,\tfrac{1}{4}\rho(x)\right)\cap D\!\left(y,\tfrac{1}{4}\rho(y)\right)
=\emptyset\quad  \text{for}\ x,y\in L, \;x\neq y.
\]
For $x\in L$, let $A_1(x),A_2(x),\ldots,A_{N(x)}(x)$ be the
sets of diameter at most $d(x)$ which cover $D(x,\rho(x))\cap K$ so
that $N(x)h(d(x))\leq \varepsilon\cdot\rho(x)^m$. Then
\[
K\subset \bigcup_{x\in L}\bigcup _{j=1}^{N(x)}A_j(x).
\]
Now
\[
\sum_{x\in L}\sum^{N(x)}_{j=1}h\!\left(\diam A_j(x)\right)
\leq
\sum_{x\in L} N(x) h(d(x))\leq \varepsilon \sum_{x\in L} \rho(x)^m.
\]
Since $\rho(x)\leq \delta$ we have
$D\left(x,\tfrac{1}{4}\rho(x)\right)\subset
D\left(0,R+\tfrac{1}{4}\delta\right)$ for all $x\in L$. Since the balls
$D\left(x,\tfrac{1}{4}\rho(x)\right)$, $x\in L$, are pairwise disjoint,
this yields
\[
\sum_{x\in L}\left(\tfrac{1}{4}\rho(x)\right)^m\leq
\left(R+\tfrac{1}{4}\delta\right)^m.
\]
We obtain
\[
H_\delta^h(K)\leq
\sum_{x\in L}\sum^{N(x)}_{j=1}
h\!\left(\diam A_j(x)\right)
\leq\varepsilon\cdot(4R+\delta)^m.
\]
Since $\delta,\varepsilon>0$ were arbitrary, we conclude that $\HH^h(K)=0$.
As this holds for every bounded subset $K$ of $A$ we deduce that $\HH^h(A)=0$.
\end{proof}

\section{Fractional iterates}\label{fractional}
It is classical (see, e.g.\ \cite[p.~670]{Ritt1926}) that if an entire function $f$
has a repelling fixed point $\xi$ of multiplier $\mu$, then the normalized solution
$S$ of Schr\"oder's functional equation $f(S(z))=S(\mu z)$ is given by
\begin{equation} \label{ritt}
S(z) =\lim_{n\to\infty} f^n(\xi +z/\mu^n).
\end{equation}
For completeness we include a proof how~\eqref{ritt} can be deduced from the
more common formula for the conjugacies near attracting fixed points.
If $f$ has an attracting fixed point at~$0$ of multiplier $\mu$, then
(see, e.g.,~\cite[Section 3.4]{Steinmetz1993})
\[
\phi(z)=\lim_{n\to\infty}\frac{f^n(z)}{\mu^n}
\]
exists and satisfies $\phi(f(z))=\mu \phi(z)$ in a neighborhood of~$0$.
Moreover, $\phi(0)=0$ and $\phi'(0)=1$.
A fixed point at $\xi\in\C$ is reduced to the case $\xi=0$ by conjugating
with $z\mapsto z-\xi$ and the case of a repelling fixed point is reduced to
this case by considering a local inverse of $f$ instead of~$f$.

Hence if $f$ has a repelling fixed point $\xi$ of multiplier $\mu$, then
it follows with $T(z)=z+\xi$ that
$T^{-1}\circ f^{-1}\circ T$ has an attracting fixed point of multiplier $1/\mu$ at~$0$ and
\[
\phi(z)=\lim_{n\to\infty}\mu^n (T^{-1}\circ f^{-1}\circ T)^n(z)
\]
converges in some neighborhood of~$0$ and satisfies
$(\phi\circ T^{-1}\circ f^{-1}\circ T)(z)=\phi(z)/\mu$ there.
Here $f^{-1}$ is a branch of the inverse fixing $\xi$, defined in some neighborhood of~$\xi$.
It follows that
\[
\phi^{-1}(z)=\lim_{n\to\infty}(T^{-1}\circ f\circ T)^n(z/\mu^n)
\]
and $(T^{-1}\circ f\circ T\circ \phi^{-1})(z)=\phi^{-1}(\mu z)$.
Let $S=T\circ \phi^{-1}$, defined in some neighborhood of~$0$.  Then $f(S(z))=S(\mu z)$ and
\[
S(z)=\lim_{n\to\infty} (f^n\circ T)(z/\mu^n) =\lim_{n\to\infty} f^n(\xi +z/\mu^n).
\]
Once it is known that this holds in a neighborhood of $0$, it follows that this
in fact holds for all~$z\in\C$.

In our case we have $f=E$ and $\xi=\mu=\beta$. Thus
\begin{equation}\label{limitS}
S(z) =\lim_{n\to\infty} E^n(\beta +z/\beta^n).
\end{equation}
It follows from this equation that the coefficients in the Taylor series expansion of $S$
are non-negative. (This can also be shown by comparing coefficients in~\eqref{schroeder}.)
We deduce that if $C>1$, then $S(Cx)/S(x)\to\infty$ as $x\to\infty$.
Hence the fractional iterates defined by~\eqref{fraciter} satisfy
\[
\lim_{x\to\infty}\frac{E^r(x)}{x}
= \lim_{x\to\infty}\frac{S(\beta^rS^{-1}(x))}{S(S^{-1}(x))}
= \lim_{y\to\infty}\frac{S(\beta^r y)}{S(y)}
=\infty
\quad\text{for }r>0.
\]
In terms of  $L^r=E^{-r}$ this takes the form
\begin{equation}\label{Lr}
\lim_{x\to\infty}\frac{L^r(x)}{x}=0
\quad\text{for }r>0.
\end{equation}
More generally,
\begin{equation}\label{Lsr}
\lim_{x\to\infty}\frac{L^s(x)}{L^r(x)}=
\lim_{x\to\infty}\frac{L^{s-r}(L^r(x))}{L^r(x)}=
\lim_{y\to\infty}\frac{L^{s-r}(y)}{y}=0
\quad\text{for }s>r.
\end{equation}
We may replace $L^s(x)$ by some power of $L^s(x)$ here.
In fact, if $\gamma>0$ and $s>r$, then
\[
L(L^s(x)^\gamma)=\gamma\log L^s(x)-\log\lambda=\gamma L^{s+1}(x)+(\gamma-1)\log\lambda
\leq L^{r+1}(x)=L(L^r(x))
\]
and hence $L^s(x)^\gamma\leq L^r(x)$  for large $x$ by~\eqref{Lsr}.
As this holds for all $\gamma>0$, we conclude that in fact
\[
\lim_{x\to\infty}\frac{L^s(x)^\gamma}{L^r(x)}=0
\quad\text{for }s>r\text{ and }\gamma>0.
\]
One consequence of this is that the conclusion of Theorem~\ref{thm1} also holds
with $h(t)=t/L^s(1/t)$ replaced by $h(t)=t/L^s(1/t)^\gamma$ if $\gamma>0$.

\begin{lemma}\label{lemma-concave}
$L^r$ is concave for $r>0$.
\end{lemma}
\begin{proof}
We show that $(L^r)'$ is non-increasing. Note that
\[
(L^r)'(u)= \beta^{-r}S'(\beta^{-r}(S^{-1})(u))(S^{-1})'(u)
=\beta^{-r}\frac{S'(\beta^{-r}x)}{S'(x)}
\]
with $u=S(x)$. Using~\eqref{limitS} we have
\[
\frac{S'(\beta^{-r}x)}{S'(x)}
=\lim_{n\to\infty}\frac{(E^n)'(\beta+\beta^{-r-n}x)}{(E^n)'(\beta+\beta^{-n}x)}
=\lim_{n\to\infty}\prod_{k=1}^nF_k(x),
\]
where
\[
F_k(x):=\frac{E^k(\beta+\beta^{-r-n}x)}{E^k(\beta+\beta^{-n}x)}.
\]
Taking the logarithmic derivative of $F_k$ gives
\[
(\log F_k)'(x)=\beta^{-r-n}\prod_{j=1}^{k-1}E^j(\beta+\beta^{-r-n}x)
-\beta^{-n}\prod_{j=1}^{k-1}E^j(\beta+\beta^{-n}x)< 0,
\]
since each factor in the second product is larger than the corresponding factor in the first product.
This implies that $F_k(x)$ is decreasing, so $S'(\beta^{-r}x)/S'(x)$ is non-increasing.
Therefore $(L^r)'$ is decreasing and thus $L^r$ is concave.
\end{proof}
\begin{lemma}\label{lemma-L(cx)}
Let $c>1$ and $r>0$. Then $L^r(cx)<cL^r(x)$ for all $x\in [\alpha,\infty)$.
\end{lemma}
\begin{proof}
Since $L^r$ is concave by Lemma~\ref{lemma-concave}, we have
\[
L^r(x)\geq \frac{x-\alpha}{cx-\alpha}L^r(cx)+\frac{cx-x}{cx-\alpha}L^r(\alpha).
\]
Since $L^r(\alpha)=\alpha$ this yields
\[
L^r(cx) \leq \frac{cx-\alpha}{x-\alpha} L^r(x) -\frac{cx-x}{x-\alpha}\alpha
= c L^r(x) +\frac{(c-1)}{x-\alpha} \left( \alpha L^r(x) -\alpha x\right).
\]
Since $L^r(x)\leq x$, the conclusion follows.
\end{proof}

\section{Hairs and endpoints: proof of Theorems~\ref{thm5} and~\ref{thm4}}\label{hairs}
We recall some results concerning the hairs that form the Julia set of~$E$.
We first note that the half-plane $\{z\in\C\colon \re z< \beta\}$ is contained in the
attracting basin of~$\alpha$.
Thus the Julia set $\JJ(E)$ is contained in the half-plane $H=\{z\in\C\colon\re z\geq\beta\}$. For $k\in\Z$ let
$$P(k)=\{z\in H\colon (2k-1)\pi\leq \im z<(2k+1)\pi\}.$$
The \emph{itinerary} of a point $z\in\JJ(E)$ is defined to be the sequence
$\underline{s}=(s_0,s_1,s_2,\cdots)$ such that $s_j=k$ if $E^{j}(z)\in P(k)$.

A sequence  $\underline{s}$ is called \emph{allowable}, if there exists $t\in \mathbb{R}$ such that
$E^j(t)\geq (2|s_j|+1)\pi$ for all $j\geq 0$. The key result proved by Devaney and Krych~\cite{Devaney1984}
(see also~\cite[Proposition~3.2]{Devaney1987}) is that if $\underline{s}$ is an allowable sequence,
then the set of all $z$ with itinerary $\underline{s}$ is a hair.
(For non-allowable $\underline{s}$ this set is empty.)

For a more detailed description we follow the ideas of Schleicher and Zimmer~\cite{Schleicher2003}
who defined the hairs (which they call dynamics rays) by using the comparison function
$F\colon [0,\infty)\to [0,\infty)$, $F(t)=e^t-1$.
Schleicher and Zimmer wrote the itineraries (which they call external addresses) in the
form $\underline{s}=(s_1,s_2,s_3,\cdots)$ instead of the notation
$\underline{s}=(s_0,s_1,s_2,\cdots)$ that was used in~\cite{Devaney1987,Devaney1984}
and that we will also use. We write their result using our terminology.

First we note that is easy to see that $\underline{s}$ is allowable
if and only if there exists $t>0$ such that
\begin{equation}\label{allowable}
\limsup_{k\to\infty}\frac{|s_k|}{F^{k}(t)}<\infty.
\end{equation}
Schleicher and Zimmer called such sequences exponentially bounded.
Let $t_{\underline{s}}$ be the infimum of all $t>0$ for which~\eqref{allowable} holds.
They then give a parametrization $g_{\underline{s}}\colon [t_{\underline{s}},\infty)\to P(s_0)$
of the hairs which satisfies
\begin{equation}\label{sz0}
E(g_{\underline{s}}(t))=g_{\sigma(\underline{s})}(F(t)) \quad \text{for}\,\,t>t_{\underline{s}},
\end{equation}
where $\sigma$ is the shift map; that is, $\sigma((s_0,s_1,s_2,\dots))=(s_1,s_2,s_3,\dots)$.
Here the function $g_{\underline{s}}$ is obtained as a limit of the functions
\[
g_{\underline{s},\,k}(t)=(L_{s_0}\circ L_{s_1}\circ \cdots\circ L_{s_k}\circ F^{k+1})(t).
\]
In~\cite[Proposition~3.4]{Schleicher2003} the convergence of this sequence is shown for $t\geq t^*$
with some $t^*\in\R$, but for the parameter range of $\lambda$ considered here it
actually holds for $t>t_{\underline{s}}$; see also~\cite[Lemma~3.1]{Bergweiler2010}.

While the results contained in the papers mentioned above are very close to the results we need,
they are not quite stated in a way suitable for us.
Thus we now describe the construction of the hairs in more detail.
First we mention that Schleicher and Zimmer noted that the choice of the
comparison function $F(t)$ is largely arbitrary.
One advantage of the function $F$ is that it does not depend on the parameter~$\lambda$.
For us it will be more convenient to use $E\colon [\beta,\infty)\to [\beta,\infty)$ instead of $F$.

To rewrite the results of Schleicher and Zimmer with this comparison function,
we consider the itinerary $\underline{s}=\underline{0}=(0,0,0,\cdots)$. By~\eqref{sz0} we have
\begin{equation}\label{sz3}
E(g_{\underline{0}}(t))=g_{\underline{0}}(F(t))\quad \text{for } t>t_{\underline{0}}=0.
\end{equation}
Here $g_{\underline{0}}\colon (0,\infty)\to\mathbb{R}$ is increasing, so  that
$\lim_{t\to 0}g_{\underline{0}}(t)$ exists and is  equal to the repelling fixed point $\beta$ of $E$.
Thus $g_{\underline{0}}$ extends to a continuous and bijective function $g_{\underline{0}}\colon [0,\infty)\to[\beta,\infty)$, and~\eqref{sz3} also holds for
$t=0$ by continuity. With $u=g_{\underline{0}}(t)$ we can write~\eqref{sz3} as
\begin{equation}\label{sz4}
g_{\underline{0}}^{-1}(E(u))=F(g_{\underline{0}}^{-1}(u)),\quad u\geq \beta.
\end{equation}
Set $u_{\underline{s}}=g_{\underline{0}}(t_{\underline{s}})$.
It follows from~\eqref{sz3} and~\eqref{sz4} that
$$u_{\underline{s}}=\inf\left\{u\in[\beta,\infty)\colon
 \limsup_{k\to\infty}\frac{|s_k|}{E^{k}(u)}<\infty\right\}.$$
The map
$h_{\underline{s}}
=g_{\underline{s}}\circ g_{\underline{0}}^{-1}\colon (u_{\underline{s}},\infty)\to \mathbb{C}$
is just a reparametrization of the hairs $g_{\underline{s}}$.
Furthermore, it follows from~\eqref{sz0} and~\eqref{sz3} that for
$u=g_{\underline{0}}(t) >g_{\underline{0}}(t_{\underline{s}})=u_{\underline{s}}$ we have
\[
E(h_{\underline{s}}(u))=E(g_{\underline{s}}(t))=g_{\sigma(\underline{s})}(F(t))
=g_{\sigma(\underline{s})}(F(g_{\underline{0}}^{-1}(u)))=h_{\sigma(\underline{s})}(E(u)).
\]
The function $h_{\underline{s}}$ is limit of the functions
\[
h_{\underline{s},\,n}(u)=(L_{s_0}\circ L_{s_1}\circ \cdots\circ L_{s_n}\circ E^{n+1})(u),
\]
that is, we have
\[
\lim_{n\to\infty} h_{\underline{s},n}(u)= h_{\underline{s}}(u) \quad\text{for } u>u_{\underline{s}}.
\]
It is proved in the papers cited above that the function $h_{\underline{s}}$ has a
continuous extension $h_{\underline{s}}\colon [u_{\underline{s}},\infty)\to\C$.
The point $h_{\underline{s}}(u_{\underline{s}})$ is then the endpoint of the hair.

The functions $h_{\underline{s},\,n}$ obviously satisfy
\begin{equation}\label{sz8}
E(h_{\underline{s},n}(u))= h_{\sigma(\underline{s}),n-1}(E(u))
\end{equation}
and taking the limit as $n\to\infty$ yields the equation
\begin{equation}\label{sz9}
E(h_{\underline{s}}(u))= h_{\sigma(\underline{s})}(E(u)) \quad\text{for } u>u_{\underline{s}}
\end{equation}
mentioned above.

\begin{lemma}\label{hsn}
For $u\geq\beta$ we have
\[
u\leq \re h_{\underline{s},\,n}(u) \leq u +\pi \sum_{k=1}^n \frac{2|s_k|+1}{\beta^{k-1}E^{k}(u)}.
\]
\end{lemma}
\begin{proof}
First we note that if $s\in\Z$ and $z\in H$, then
\begin{equation}\label{hs1}
L(|z|)\leq |L_s(z)|\leq L(|z|)+(2|s|+1)\pi.
\end{equation}
The left inequality of~\eqref{hs1} implies that
\[
\re h_{\underline{s},\,n}(u)
=
L\!\left(|(L_{s_{1}}\circ \dots\circ L_{s_n})(E^{n+1}(u))|\right)\geq L^{n+1}(E^{n+1}(u))=u.
\]
To prove the upper bound for $\re h_{\underline{s},\,n}(u)$ we note that if $x\geq \beta$ and $y>0$, then
\[
L(x+y)
=L\!\left(x\!\left(1+\frac{y}{x}\right)\right)=L(x)+\log\left(1+\frac{y}{x}\right)\leq L(x)+\frac{y}{x}.
\]
We conclude that
\[
L^2(x+y)
\leq L\!\left(  L(x)+\frac{y}{x}\right) \leq L^2(x)+\frac{y}{x L(x)}
\leq L^2(x)+\frac{y}{\beta x}
\]
and induction shows that
\begin{equation}\label{hs2}
L^k(x+y)\leq L^k(x)+\frac{y}{\beta^{k-1}x}
\end{equation}
for $k\in\N$.  We now fix $n\in\N$ and, for $1\leq k\leq n$, put
\[
p_k(u)=L^k \!\left(|(L_{s_{k}}\circ \dots\circ L_{s_n})(E^{n+1}(u))|\right).
\]
Then
\[
p_1(u)= \re h_{\underline{s},\,n}(u)
\]
and for $1\leq k\leq n-1$ we find, using~\eqref{hs1} and~\eqref{hs2}, that
\begin{align*}
& \quad \  p_k(u)
\\
& =
L^k \!\left(|(L_{s_{k}}\circ \dots\circ L_{s_n})(E^{n+1}(u))|\right)\\
&\leq
L^k \!\left(L(|(L_{s_{k+1}}\circ \dots\circ L_{s_n})(E^{n+1}(u))|)+(2|s_k|+1)\pi\right)
\\ &
\leq
L^k \!\left( L(|(L_{s_{k+1}}\circ \dots\circ L_{s_n})(E^{n+1}(u))|)\right)
+\frac{(2|s_k|+1)\pi}{\beta^{k-1}L(|(L_{s_{k+1}}\circ \dots\circ L_{s_n})(E^{n+1}(u))|)}
\\ &
\leq
L^{k+1} \!\left( |(L_{s_{k+1}}\circ \dots\circ L_{s_n})(E^{n+1}(u))|\right)
+\frac{(2|s_k|+1)\pi}{\beta^{k-1}|L^{n-k+1}(E^{n+1}(u))|}
\\ &
=
L^{k+1} \!\left( |(L_{s_{k+1}}\circ \dots\circ L_{s_n})(E^{n+1}(u))|\right)
+\frac{(2|s_k|+1)\pi}{\beta^{k-1}E^{k}(u)}
\\ &
=
p_{k+1}(u)
+\frac{(2|s_k|+1)\pi}{\beta^{k-1}E^{k}(u)}.
\end{align*}
This also holds for $k=n$ with $p_{n+1}(u)=L^{n+1}(E^{n+1}(u))=u$. We obtain
\[
\re h_{\underline{s},\,n}(u)
= p_1(u)
= u
+\sum_{k=1}^n  \left(p_k(u)-p_{k+1}(u)\right)
\leq
u
+\pi \sum_{k=1}^n \frac{2|s_k|+1}{\beta^{k-1}E^{k}(u)}
\]
as claimed.
\end{proof}
\begin{proof}[Proof of Theorem~\ref{thm5}]
Let $z\in\JJ(E)\backslash \CC$ and let $\underline{s}$ be the itinerary of~$z$.
Then $z=h_{\underline{s}}(u)$ for some $u>u_{\underline{s}}$.
It follows from~\eqref{sz9} and Lemma~\ref{hsn} that
\[
\re E^k(z)=\re E^k(h_{\underline{s}}(u))=
\re h_{\sigma^k(\underline{s})}(E^k(u))\geq E^k(u)
\]
for all $k\in\N$. Let $u_{\underline{s}}<v<u$. For large $k$ we then have
\[
|\im E^k(z)|\leq (2|s_k|+1)\pi < E^k(v).
\]
Since $v<u$ we have $v=L^\varepsilon(u)$ for some $\varepsilon>0$.
Altogether we see that
\[
|\im E^k(z)|<  E^k(L^\varepsilon(u))=E^{k-\varepsilon}(u)=L^\varepsilon(E^k(u))
\leq L^\varepsilon(\re E^k(z))
\]
for large~$k$. Since also $\re E^k(u)>x_0$ for large~$k$, the conclusion follows.
\end{proof}
\begin{proof}[Proof of Theorem~\ref{thm4}]
Since $\FF(E)$ consists of the attracting basin of $\alpha$ we have $\XX(x_0,L^\varepsilon)\subset \JJ(E)$.
Let $z\in\CC$. Then $z=h_{\underline{s}}(u_{\underline{s}})$ for some
allowable sequence~$\underline{s}$. Let $u> u_{\underline{s}}$.
Then there exists $k_0\in\N$ such that $E^k(u)\geq (2|s_k|+1)\pi$ for $k\geq k_0$.
For $u>u_{\underline{s}}$ and  $n\geq k\geq k_0$ we deduce from~\eqref{sz8} and Lemma~\ref{hsn} that
\begin{align*}
\re E^k(h_{\underline{s},n}(u))
& = \re h_{\sigma^k(\underline{s}),n-k}(E^k(u))
\leq E^k(u)
+\pi \sum_{j=1}^{n-k} \frac{2|s_{k+j}|+1}{\beta^{j-1}E^{j}(E^k(u))}
\\ &
= E^k(u)
+\pi \sum_{j=1}^{n-k} \frac{2|s_{k+j}|+1}{\beta^{j-1}E^{j+k}(u)}
\leq E^k(u)
+\sum_{j=1}^{n-k} \frac{1}{\beta^{j-1}}
\\ &
\leq E^k(u)
+\frac{\beta}{\beta-1} .
\end{align*}
It follows that
\begin{equation}\label{hs3}
\re E^k(h_{\underline{s}}(u)) \leq E^k(u) +\frac{\beta}{\beta-1}
\end{equation}
for $u>u_{\underline{s}}$ and  $k\geq k_0$.

If $u_{\underline{s}}=\beta$, choose $\gamma\in(\beta,x_0)$.
Then~\eqref{hs3}  holds in particular for all $u\in (u_{\underline{s}},\gamma)$ and
thus, by continuity, also  for $u=u_{\underline{s}}$.  Hence
\[
\re E^k(z)= \re E^k(h_{\underline{s}}(u_{\underline{s}}))
 \leq E^k(\gamma) +\frac{\beta}{\beta-1}
\leq E^k(x_0)
\]
for large $k$, which implies that $z\notin\XX(x_0,L^\varepsilon)$.

Suppose now that $u_{\underline{s}}>\beta$. Choose $u_1,u_2$ with $\beta<u_1<
u_{\underline{s}}<u_2$ such that $u_2<E^\varepsilon (u_1)$.
We obtain
\[
\re E^k(z) \leq E^k(u_2) +\frac{\beta}{\beta-1} \leq 2 E^k(u_2)
\]
for large $k$, while
\[
\im E^k(z) \geq (2|s_k|-1) |\pi|\geq 2 E^k(u_1)
\]
for arbitrarily large~$k$.  Using Lemma~\ref{lemma-L(cx)} we  obtain
\begin{align*}
\psi(\re E^k(z))
&=L^\varepsilon (\re E^k(z))
\leq L^\varepsilon (2 E^k(u_2))
\\ &
\leq 2 L^\varepsilon (E^k(u_2))
= 2 E^k(L^\varepsilon(u_2))
\\ &
= 2 E^k(u_1)
\leq \im E^k(z)
\end{align*}
and hence $E^k(z)\notin \Omega$ for arbitrarily large~$k$.  Thus $z\notin\XX(x_0,L^\varepsilon)$ also in this case.
\end{proof}

\section{Proof of Theorem~\ref{thm3}}\label{proofthm3}
The proof of $(i)$ will follow the arguments of Karpi\'{n}ska and Urba\'{n}ski~\cite{Karpinska2006}
while the proof of $(ii)$ will also use some ideas from~\cite{Bergweiler2010}.
\begin{proof}[Proof of Theorem~\ref{thm3}]
Following~\cite{Karpinska2006} we consider
the family $\BB$ of all squares of the form
\[
\left\{z\in \C\colon \beta+ (l-1)\pi\leq \re z< \beta +l\pi
\text{ and }
-\tfrac{1}{2}\pi+2k\pi\leq \im z\leq\tfrac{1}{2}\pi+2k\pi\right\}
\]
with $k\in\Z$ and $l\in\N$.  Since $\JJ(E)\subset \{z\in\C\colon \re z\geq\beta\}$, we see that
for every $z\in \JJ(E)$ and $n\geq 0$ there exists a unique square $B_n(z)\in \mathcal{B}$ such
that $E^n(z)\in B_n(z)$.
We denote by $K_n(z)$ the component of $E^{-n}(B_n(z))$ that contains~$z$.
Then $E^n(K_{n-1}(z))$ is a half-annulus centered at the origin.
We denote its inner and outer radius by $r_n(z)$ and $R_n(z)$. Clearly $R_n(z)/r_n(z)=e^{\pi}$.

For $n\geq 0$ we define a collection $\KK_n$ of sets $K_n(z)$ by recursion.
First we choose $z_0\in\C$ with $\re z_0>x_0$ such that the square
$B_0(z_0)$ is contained in $\Omega_\psi$ and we put $\KK_0=\{B_0(z_0)\}$.
Assuming that $\KK_{n-1}$ has been defined, let $K=K_{n-1}(z)\in \KK_{n-1}$.
Then $E^{n-1}(K)\in\BB$ and thus $E^n(K)$ is a half-annulus.
The sets $K_{n}(\zeta)$ which are contained in $K$ and which have the
property that
\[
E^n(K_n(\zeta))\subset \left\{ z\in\C\colon 2r_n(z)<\re z <\tfrac23 R_n(z)
\text{ and } |\im z|< \psi(r_n(z))\right\}
\]
are called the children of $K$.  The set of children of $K$ is denoted by $ch(K)$.  We then put
\[
\KK_n=\bigcup_{K\in \KK_{n-1}} ch(K) .
\]

Let $X_n$ be the closure of the union of the elements in $\KK_n$.
We define a sequence $(\mu_n)$ of measures with $\supp\mu_n=X_n$ by recursion.
Let  $\mu_0$ be the normalized Lebesgue measure on $X_0$; that is,
$\mu_0(A)=\area(A\cap X_0)/\pi^2$
for every measurable subset~$A$ of $\C$.
Suppose now that the measure $\mu_n$ on $X_n$ has been defined. The measure
$\mu_{n+1}$ on $X_{n+1}$ is then defined on each $K_{n+1}\in \KK_{n+1}$ by
\[
\mu_{n+1}|_{K_{n+1}}=\frac{\area K_{n+1} }{\sum_{K\in ch(K_n)}\area K }\mu_n|_{K_n},
\]
where $K_n$ is the unique element of $\KK_n$ containing $K_{n+1}$.
Then (see~\cite{Karpinska2006} for more details)
there exists a unique Borel measure $\mu$ on $X_\infty=\bigcap_{n=0}^\infty X_n$ which
satisfies $\mu(K_n)=\mu_n(K_n)$ for every $K_n\in \KK_n$.

Karpi\'{n}ska and Urba\'{n}ski showed that there exists constants $\eta>0$ and $L,M>1$ such that if
$\re z$ is large enough, then~\cite[Lemma~2.3]{Karpinska2006}
\begin{equation} \label{ku0}
r_{n+1}(z)\geq \exp(\eta r_n(z))
\end{equation}
as well as~\cite[equation~(8)]{Karpinska2006}
\begin{equation} \label{ku1}
M^{-n}\prod_{j=1}^n r_j(z)\leq |(E^n)'(z)|\leq M^n\prod_{j=1}^n r_j(z)
\end{equation}
and~\cite[equation~(16)]{Karpinska2006}
\begin{equation}
\label{ku2}
\begin{aligned}
\mu(K_n(z))
&= c_1(z)^{n}\left( \prod_{j=1}^n r_j(z) \right)^{-2}\prod_{i=1}^n\frac{r_i(z)^2}{r_i(z)\psi(r_i(z))}
\\ &
=c_1(z)^n \prod_{i=1}^n\frac{1}{r_i(z)\psi(r_i(z))}
\end{aligned}
\end{equation}
for some $c_1(z)\in [L^{-1},L]$.

It follows easily from~\eqref{ku0} that if $\re z$ is large enough,
then $\re E^k(z)>E^k(x_0)$ for all $k\in\N$.
Hence $X_\infty\subset \XX$ if $z_0$ was chosen with $\re z_0$ is sufficiently large.
It thus suffices to show that $\HH^h(X_\infty)=\infty$.

In order to apply Lemma~\ref{th21}, we have to estimate $\mu(D(z,t))$ for $z\in X_\infty$.
We fix $z\in X_\infty$ and write $r_j$ and $R_j$ instead of $r_j(z)$ and $R_j(z)$.
It follows from~\eqref{ku2} that
\begin{equation} \label{ku6}
\mu(K)\leq L^n \prod_{i=1}^n\frac{1}{r_i\psi(r_i)} \quad\text{for } K\in ch(K_{n-1}(z)).
\end{equation}
It is not difficult to deduce from~\eqref{ku0} that if $C>1$ and $\delta>0$, then
\begin{equation} \label{ku6a}
C^n \prod_{i=1}^{n-1} r_i  \leq (\log r_n)^{1+\delta}
\end{equation}
for large~$n$.

Since $E^n(z)$ is near the ``center'' of the half-annulus $E^n(K_{n-1}(z))$
and thus $E^{n-1}(z)$ is near the center of the square $E^{n-1}(K_{n-1}(z))=B_{n-1}(z)$,
the Koebe distortion theorem yields that there exist $c>0$ such that
\begin{equation} \label{ku6b}
D(z,c/|(E^{n-1})'(z)|) \subset K_{n-1}(z) .
\end{equation}

For small $t>0$ we choose $n\in\N$ such that $c/|(E^n)'(z)|<t\leq c/|(E^{n-1})'(z)|$.
Hence $D(z,t)\subset K_{n-1}(z)$.
Denoting by $N(z,t)$ the number of children of $K_{n-1}(z)$ which intersect $D(z,t)$
we deduce from~\eqref{ku6} that
\begin{equation} \label{ku5}
\mu(D(z,t))\leq N(z,t) L^n \prod_{i=1}^n\frac{1}{r_i\psi(r_i)} .
\end{equation}

Using again Koebe's theorem, we see that $E^n(D(z,t))$  is contained in disk around $E^n(z)$ of radius
$C|(E^n)'(z)|t$ for some constant~$C$.  This implies that $N(z,r)$ is bounded by the number of elements
of $\BB$ contained in the intersection of the disk $D(E^n(z),C|(E^n)'(z)|t+2\pi)$
with the strip $T_n=\{z\in\C\colon |\im z|\leq \psi(r_n)\}$.  We may assume that $C\geq 2\pi/c$.
Noting that $|(E^n)'(z)|t\geq c$ by the choice of $n$ we see that
$C|(E^n)'(z)|t+2\pi\leq 2C|(E^n)'(z)|t$.  Thus
\begin{equation} \label{ku7}
N(z,r)\leq \area\!\left( D(E^n(z),2C|(E^n)'(z)|t) \cap T_n \right).
\end{equation}
We distinguish two cases:

\emph{Case} 1: $2C|(E^n)'(z)|t\leq \psi(r_n)$.  Then we may simplify~\eqref{ku7} to
\[
N(z,t)\leq \area D(E^n(z),2C|(E^n)'(z)|t) =4\pi C^2 |(E^n)'(z)|^2 t^2
\]
and thus~\eqref{ku1} and~\eqref{ku5} yield
\begin{equation} \label{ku9}
\mu(D(z,t))\leq 4\pi C^2 (M^2L)^n t^2 \prod_{i=1}^n\frac{r_i}{\psi(r_i)} .
\end{equation}
Since $t\mapsto t\,p(1/t)$ is increasing and
\[
t\leq \frac{\psi(r_n)}{2C|(E^n)'(z)|}\leq \frac{M^n}{2C}\frac{\psi(r_n)}{\prod_{j=1}^n r_j}
\]
by~\eqref{ku1} and the choice of~$n$, we deduce from~\eqref{ku9} and~\eqref{ku6a} that
\begin{align*}
\frac{\mu(D(z,t))}{h(t)}
&=\frac{\mu(D(z,t))}{t}p\!\left(\frac{1}{t}\right)
\leq 4\pi C^2 (M^2L)^n \,t\, p\!\left(\frac{1}{t}\right)  \prod_{i=1}^n\frac{r_i}{\psi(r_i)}
\\ &
\leq 4\pi C^2 (M^2L)^n \frac{M^n}{2C}\frac{\psi(r_n)}{\prod_{j=1}^n r_j}
\, p\!\left(\frac{2C}{M^n}\frac{\prod_{j=1}^n r_j}{\psi(r_n)}\right)  \prod_{i=1}^n\frac{r_i}{\psi(r_i)}
\\ &
= 2\pi C (M^3L)^n
\, p\!\left( \frac{2C}{M^n}\frac{\prod_{j=1}^n r_j}{\psi(r_n)} \right) \prod_{i=1}^{n-1}\frac{1}{\psi(r_i)}
\\ &
\leq p\!\left( \frac{r_n (\log r_n)^{1+\delta}}{\psi(r_n)} \right) \frac{1}{\psi(r_{n-2})\psi(r_{n-1})}
\end{align*}
for large~$n$.  Since $r_n\leq E(R_{n-1})\leq \exp(R_{n-1})$ we have 
$\log r_n\leq R_{n-1}\leq e^\pi r_{n-1}$
and thus $\psi(\log r_n)\leq \psi(e^\pi r_{n-1})=O(\psi(r_{n-1}))$.  It now follows from~\eqref{condp}
that $\mu(D(z,t))/h(t)=O( 1/\psi(r_{n-2}))$. We conclude that $\mu(D(z,t))/h(t)\to 0$ in this case.

\emph{Case} 2: $2C|(E^n)'(z)|t> \psi(r_n)$.
Then
\begin{align*}
N(z,t) &\leq \area \left\{\zeta \in\C \colon |\re (\zeta-E^n(z)|\leq 2C|(E^n)'(z)|t
\text{ and } |\im z|\leq \psi(r_n)\right\}
\\ &
= 8C|(E^n)'(z)|\psi(r_n) t
\end{align*}
which together with~\eqref{ku1} and~\eqref{ku5} yields
\begin{equation} \label{ku11}
\mu(D(z,t)) \leq  8C(ML)^n \, t  \prod_{i=1}^{n-1}\frac{1}{\psi(r_i)}.
\end{equation}
Now
\[
t\geq \frac{\psi(r_n)}{2C|(E^n)'(z)|}\geq \frac{1}{2CM^n}\frac{\psi(r_n)}{\prod_{j=1}^n r_j}
\]
and hence~ \eqref{ku11} yields
\begin{align*}
\frac{\mu(D(z,t))}{h(t)}
&=\frac{\mu(D(z,t))}{t}\, p\!\left(\frac{1}{t}\right)
\\ &
\leq 8C(ML)^n  \, p\!\left(2C M^{n}\frac{\prod_{j=1}^{n} r_j}{\psi(r_n)} \right)
 \prod_{i=1}^{n-1}\frac{1}{\psi(r_i)}
\\ &
\leq p\!\left( \frac{r_n (\log r_n)^{1+\delta}}{\psi(r_n)} \right) \frac{1}{\psi(r_{n-2})\psi(r_{n-1})}
\end{align*}
from which we obtain $\mu(D(z,t))/h(t)\to 0$ as in Case~1.

Thus we have $\mu(D(z,t))/h(t)\to 0$ as $t\to 0$ in both cases and $(i)$ follows
from Lemma~\ref{th21}.

In order to prove $(ii)$,
we denote for $z\in\XX$ by $B_n^*(z)$ the square of sidelength $\pi$ with
center~$z$, with sides parallel to the coordinate axes.
We denote by $K_n^*(z)$ the component of $E^{-n}(B_n^*(z))$ that contains~$z$.
Then $E^n(K_{n-1}^*(z))$ is again a half-annulus centered at the origin.
We denote its inner and outer radius by $r_n^*(z)$ and $R_n^*(z)$.
As before we have $R_n^*(z)/r_n^*(z)=e^{\pi}$. With the quantities $r_n(z)$ and $R_n(z)$
defined in the first part of the proof we have $e^{-\pi}\leq r_n^*(z)/r_n(z)\leq e^{\pi}$ and
$e^{-\pi}\leq R_n^*(z)/R_n(z)\leq e^{\pi}$.
In particular, \eqref{ku1} holds for some $M>1$ and~\eqref{ku6a} holds for any
given $C>1$ with $r_n(z)$ and $R_n(z)$ replaced by $r_n^*(z)$ and $R_n^*(z)$,
provided $n$ is sufficiently large.

We now fix $z\in\XX$ and, as before, drop $z$ from the notation and write
$r_n^*$ and $R_n^*$ instead of  $r_n^*(z)$ and $R_n^*(z)$.

We can cover $E^n(K_{n-1}^*(z))\cap \Omega_\psi$ by squares of sidelength
$2\psi(R_n^*)$ whose centers are on the real axis.
The number $N_n(z)$ of squares required is less than $R_n^*/\psi(R_n^*)$.
Koebe's theorem yields that the diameters of the preimages
of these squares that are contained in $K_{n-1}^*(z)$
have diameter less than $C \psi(R_n^*)/|(E^n)'(z)|$, for some constant $C>1$.
Moreover, \eqref{ku6b} holds for some $c>0$ with $K_{n-1}(z)$ replaced by $K_{n-1}^*(z)$.
With $\rho_n(z)=c/|(E^{n-1})'(z)|$ and $d_n(z)=C \psi(R_n^*)/|(E^n)'(z)|$
we thus see that $D(z,\rho_n(z))\cap \XX$ can be covered by $N_n(z)$ sets of
diameter $d_n(z)$. Moreover, $d_n(z)\to 0$ and $\rho_n(z)\to 0$ as $n\to\infty$.

Hence $(ii)$ follows from Lemma~\ref{th22} if we show that
for given $\varepsilon>0$ we have
$N_n(z)h(d_n(z))\leq \varepsilon\cdot\rho_n(z)^2$ for large~$n$. Noting that
\[
\rho_n(z)\geq c M^{-n}\prod_{j=1}^{n-1} \frac{1}{r_j^*} \geq \frac{1}{(\log R_n^*)^{1+\delta/3}}
\]
by~\eqref{ku6a}, with $\delta$ replaced by $\delta/3$, and
\[
d_n(z)\leq C M^{n} \psi(R_n^*) \prod_{j=1}^{n} \frac{1}{r_j^*}
\leq \frac{\psi(R_n^*)}{R_n^* \log R_n^*}
\]
by~\eqref{ku1}
it thus suffices to show that
\[
\frac{R_n^*}{\psi(R_n^*)}
h\!\left(\frac{\psi(R_n^*)}{R_n^* \log R_n^*}\right)
=
\frac{1}{\displaystyle p\!\left(\frac{R_n^*  \log R_n^*}{ \psi (R_n^*)}\right) \log R_n^*}
\leq \varepsilon\frac{1}{(\log R_n^*)^{2+2\delta/3}},
\]
which is equivalent to
\[
p\!\left(\frac{R_n^*  \log R_n^*}{ \psi (R_n^*)}\right)
\geq \frac{1}{\varepsilon} (\log R_n^*)^{1+2\delta/3}.
\]
But the last inequality is satisfied by the hypothesis~\eqref{condp2} for large~$n$.
\end{proof}

\section{Proof of Theorems~\ref{thm2} and~\ref{thm1}} \label{proofthm21}
\begin{proof}[Proof of Theorem~\ref{thm2}]
Let $p(t)=(\log t)^s$ so that $h(t)=t/p(1/t)$.
Let $x_0>\beta$,  $\varepsilon>0$, $0<\delta<s-1$ and $\psi=L^\varepsilon$.
For large $t$ we then have, using~\eqref{Lsr},
\begin{align*}
p\!\left(\frac{t\log t}{\psi(t)}\right)
&= \left( \log t+\log\log t-\log L^\varepsilon(t)\right)^s \\
&= \left( L(t)+\log\log t-L^{1+\varepsilon}(t)\right)^s
\geq  \left( \tfrac12 \log t\right)^s
\geq (\log t)^{1+\delta}
\end{align*}
so that~\eqref{condp2} holds. Thus Theorem~\ref{thm3}, part (ii), implies that
$\HH^h(\XX(x_0,L^\varepsilon))=0$. As $\varepsilon>0$ was arbitrary, we see that for
\[
Y=\bigcup_{\varepsilon>0} \XX(x_0,L^\varepsilon)
=\bigcup_{n=1}^\infty \XX(x_0,L^{1/n})
\]
we also have $\HH^h(Y)=0$.
Since $E$ is locally bi-Lipschitz and $h(2t)=O(h(t))$ as $t\to 0$ we see that
$\HH^h(E^{-1}(Y))=0$ and in fact that
$\HH^h(E^{-k}(Y))=0$ for all $k\in\N$.
Since
\[
\JJ\backslash\CC \subset \bigcup_{k=1}^\infty E^{-k}(Y)
\]
by Theorem~\ref{thm5}, the conclusion follows.
\end{proof}

\begin{proof}[Proof of Theorem~\ref{thm1}]
Let $0<\varepsilon<s-1$.
In view of Theorems~ \ref{thm3} and~\ref{thm4} it suffices to show
that $\psi=L^\varepsilon$ and $p=L^s$ satify the hypotheses of Theorem~\ref{thm3}.

Clearly, $\psi$ and $p$ are increasing and~\eqref{Lr} says that $\psi(x)=o(x)$ 
while Lemma~\ref{lemma-L(cx)} yields that $\psi(2x)=O(\psi(x))$
as $x\to\infty$.
In order to check the conditions on $p$ we put $q\colon(0,t_0)\to(0,\infty)$, $q(t)=t\,p(1/t)$.
Then $q''(t)=p''(1/t)/t^3$. Since $p=L^r$ is concave by Lemma~\ref{lemma-concave},
we see that $q$ is concave.
Since also $q(t)\to 0$ as $t\to 0$ by~\eqref{Lr} this implies that $q$ is increasing on a suitable
interval~$(0,t_0)$.

It remains to verify condition~\eqref{condp}.  Recalling that $L(x)=\log x-\log\lambda$
and hence $p(x)=L^{s-1}(\log x-\log\lambda)$ we deduce from Lemma~\ref{lemma-L(cx)} that
\begin{align*}
p\!\left(\frac{t(\log t)^{1+\delta}}{\psi(t)}\right)
&=L^{s-1}(\log  t+(1+\delta)\log\log t-\log \psi(t)-\log\lambda)
\\ &
\leq L^{s-1}(2t)
\leq  2L^{s-1}(t)
\end{align*}
for large~$t$.
Since $\varepsilon<s-1$ we have $L^{s-1}(t)=o(L^\varepsilon(t))$ as $t\to\infty$ by~\eqref{Lsr}.
Hence $2L^{s-1}(t)\leq L^\varepsilon(t)=\psi(t)$ for large $t$ so that~\eqref{condp} holds.
\end{proof}
\begin{remark}
Let $\psi(t)=t/(\log t)^\varepsilon$. For $p(t)=t^{1/(1+2\delta+\varepsilon)}$ we have
\[
p\!\left(\frac{t(\log t)^{1+\delta}}{\psi(t)}\right)=p((\log t)^{1+\delta+\varepsilon})
=(\log t)^{(1+\delta+\varepsilon)/(1+2\delta+\varepsilon)}
\leq \frac{\log t}{(\log\log t)^\varepsilon}=\psi(\log t)
\]
for large $t$ so that~\eqref{condp} is satisfied.
The other hypotheses of Theorem~\ref{thm3} are also easily checked.
Hence part $(i)$ of Theorem~\ref{thm3} implies that
$\HH^h(\XX)=\infty$ for $h(t)=t/p(1/t)=t^{1+1/(1+2\delta+\varepsilon)}$.
With $\delta\to 0$ we see that the Hausdorff dimension of $\XX$ is at
least $1+1/(1+\varepsilon)=(2+\varepsilon)/(1+\varepsilon)$.

Similarly, with $p(t)=t^{(1+\delta)/(1+\varepsilon)}$ we have
\[
p\!\left(\frac{t \log t}{\psi(t)}\right)=p((\log t)^{1+\varepsilon})
=(\log t)^{1+\delta}
\]
for large $t$ so that~\eqref{condp2} is satisfied.
Now part $(ii)$  of Theorem~\ref{thm3}, together with the limit as
$\delta\to 0$, implies that the Hausdorff dimension of $\XX$ is at
most $(2+\varepsilon)/(1+\varepsilon)$.

Altogether we see that $\XX$ has Hausdorff dimension $(2+\varepsilon)/(1+\varepsilon)$,
thus recovering the result of~\cite{Karpinska2006} mentioned in the introduction.
\end{remark}

\noindent{\bf Acknowledgements.}\,\,The second author was supported by the scholarship from China Scholarship Council (No.201206105015), and would express her thanks for the hospitality of Mathematisches Seminar at Christian-Albrechts-Universit\"{a}t zu Kiel.

\end{document}